\documentclass[12pt]{article}
\usepackage{amssymb}
\usepackage{amsthm}
\usepackage{latexsym,array}
\usepackage{amsthm}
\usepackage{amsmath, mathrsfs}
\usepackage[english]{babel}
\usepackage{cite}
\usepackage{color}

\newtheorem{theorem}{Theorem}[section]
\newtheorem{definition}[theorem]{Definition}

\newtheorem{proposition}[theorem]{Proposition}
\newtheorem{corollary}[theorem]{Corollary}
\newtheorem{remark}[theorem]{Remark}

\begin{document}

\title{A Borel-Pompeiu formula in a $(q,q')$-model of quaternionic analysis}
\small{
\author {Jos\'e Oscar Gonz\'alez-Cervantes$^{(1)}$, Juan Bory-Reyes$^{(2)}$,\\and\\ Irene Sabadini$^{(3)}$}
\vskip 1truecm
\date{\small $^{(1)}$ Departamento de Matem\'aticas, ESFM-Instituto Polit\'ecnico Nacional. 07338, Ciudad M\'exico, M\'exico\\ Email: jogc200678@gmail.com\\$^{(2)}$ {SEPI, ESIME-Zacatenco-Instituto Polit\'ecnico Nacional. 07338, Ciudad M\'exico, M\'exico}\\ Email: juanboryreyes@yahoo.com \\$^{(3)}$ Politecnico di Milano, Dipartimento di Matematica, Via E. Bonardi 9, 20133
Milano, Italy\\ Email: irene.sabadini@polimi.it
}
\maketitle

\begin{abstract}
\small{
The study of $\psi-$hyperholomorphic functions defined on domains in $\mathbb R^4$ with values in $\mathbb H$, namely null-solutions of the $\psi-$Fueter operator, is a topic which captured great interest in quaternionic analysis. This class of functions is more general than that of Fueter regular functions.

In the setting of $(q,q')-$calculus, also known as post quantum calculus, we introduce a deformation of the $\psi-$Fueter operator written in terms of suitable difference operators, which reduces to a deformed $q$ calculus when $q'=1$. We also prove the Stokes and  Borel-Pompeiu formulas in this context.

This work is the first investigation of results in quaternionic analysis in the setting of the $(q,q')-$calculus theory.}
\vskip 0.3truecm
\small{
\noindent
\textbf{Keywords.} $(q,q')-$calculus, {Stokes formula}, Borel-Pompeiu formula, quaternionic analysis\\
\noindent
\textbf{MSC 2020.} Primary: 30G30; 30G35, 05A30, Secondary: 46S05, 47S05}
\end{abstract}

\section{Introduction}
Quantum calculus, or $q-$calculus for short, is the counterpart of classical calculus without the notion of limits. It has gained popularity during the last two decades, and a wide variety of new results can be found in the literature.

The study of $q-$calculus, goes back to the work of Euler. However, an extensive study has been performed at the beginning of the past century by  Jackson  \cite{jackson1, jackson2, jackson3}, who introduced the notion of $q-$derivative. For a wide discussion of $q-$calculus we refer the reader to \cite{KacChe, Ernst} and the references therein.

In classical calculus, the ordinary derivative $f'(x_0)$ of a function $f(x)$ of one variable ($x\in \mathbb R$) at $x=x_0$ is given by the limit, if it exists, of the expression
\begin{align*}
\frac{f(x)-f(x_0)}{x-x_0}
\end{align*}
as $x$ approaches to $x_0$. If we take $x=qx_0$, where $q$ is a fixed number different from $1$ and do not take the limit, the corresponding expression
\begin{align*}
\frac{f(qx_0)-f(x_0)}{qx_0-x_0}
\end{align*}
 is a difference operator, which is the so-called $q-$derivative. We are now in the framework of $q$-calculus, also called quantum calculus.

Applications of $q-$calculus have been studied and investigated intensively, especially for the connections with physics. Inspired and motivated by these applications, many researchers have developed a version of this calculus based on two parameters, namely the so-called $(q, q')-$model (also called post quantum calculus), which is used efficiently in various areas of mathematics and also in quantum physics. More information and results concerning the $(q,q')-$calculus are e.g. in \cite{BurKli, Dem, NSP, Gupta2018}.

For $0 < q' < q \leq 1$, the $(q,q')-$derivative of the function $f(x)$ is defined by
\begin{align*}
D_{q,q'}f(x) := \frac{f(qx) - f(q'x)}{(q - q')x}, \quad  x\neq 0
\end{align*}
{and was first proposed by Chakrabarti and Jagannathan \cite{CJ}, as a generalization of the Jackson derivative (for $q' = 1$) of the symmetric derivative (for $q' = q^{-1}$) and finally of the McAnally derivative (for $q \rightarrow  q^{1-\lambda}, q' \rightarrow  q^{-\lambda}$, where $q$ and $\lambda$ are parameters), see  the formulation in \cite{McAnally1, McAnally2}.}

Note that $(q,q')-$derivative presents invariance under the transformation $q \leftrightarrow q'$.

In this paper we combine the study of $(q,q')$-calculus with the study of $\psi-$hy\-per\-holo\-mor\-phic functions, a topic of great interest in quaternionic analysis.
These functions are defined on domains in $\mathbb R^4$, have values in $\mathbb H$, and are null-solutions of the so-called $\psi-$Fueter operator related to a structural set $\psi$ of $\mathbb H^4$.

The study of quaternionic functions of one quaternionic variable which are "holomorphic", in a suitable sense, started in the thirties of the past century with the works of  Moisil and Fueter \cite{Moisil,Fu1}. These functions are nowadays called Cauchy-Fueter regular (or Fueter regular, for short) and they have been extensively studied by the school around Fueter. For an account of these functions the reader may consult the survey of Sudbery \cite{sudbery}, elaborated from the original sources in German, or the books \cite{CSSS,GHS,GS2} and the references therein.

Our study is a twofold extension of the previous studies in hypercomplex analysis: on one hand we are going beyond the $q$-calculus already studied in \cite{CS1,CS2,zimmermann2022steps} since we introduce here for the first time the $(q,q')$-calculus; on the other hand, we work with an operator more general than the classical Fueter operator, namely with the $\psi$-operator.

It is worthwhile to mention that the papers \cite{CS1,CS2,zimmermann2022steps} are in the Clifford analysis setting, however our study can be generalized to that { framework with the techniques used in this paper}. Furthermore, the idea of a several variables $q-$calculus (also called  multiple  $q-$calculus, see e.g. \cite{AAD}) is already formulated and introduced in the literature, see e.g. \cite{DAA, Pashaev2014ExactlySQ, AAD, ANT, FB} with no claim of completeness. This {work} opens the way to a multiple $q-$calculus in quaternionic analysis, and more in general, in Clifford analysis.

The plan of the paper is as follows. Section 2 contains a brief summary of the standard facts on quaternionic analysis and revises some basic definitions and properties of the $(q, q')-$calculus. At the end we introduce the notion of quaternionic $(q,q')-$derivative, as a natural generalization of the post quantum derivative operator. Section 3 deals with a multiple $(q,q')-$deformation of the $\psi-$Fueter operator. In the final Section 4, the associated $(q, q')-$Stokes's formula and the $(q, q')-$Borel-Pompeiu formula are obtained.

\section{Preliminaries}
\subsection{Quaternionic analysis}

Nowadays, quaternionic analysis is regarded as a broadly studied extension of classical analysis offering a successful generalization of complex analysis. It deals with the study of functions in the kernel of a generalized Cauchy-Riemann operator. In particular, when the operator is written in terms of an orthonormal basis $\psi$ different from the canonical basis, it gives the notion of $\psi-$hyperholomorphic functions. Below we repeat the basics of this function theory.
Some material is taken from \cite{shapiro1,shapiro2,Krav,KS} to which we refer the reader for more information.

Let $\mathbb H$ be the skew field of real quaternions and let $e_0=1,e_1, e_2, e_3$ be the quaternion imaginary units that fulfill the relations
\[e_ie_j+e_je_i=-2\delta_{ij},\;i,j=1,2,3\]
\[e_1e_2=e_3;\;e_2e_3=e_1;\;e_3e_1=e_2.\]

For any $x=\sum_{j=0}^{3}x_je_j\in\mathbb{H}$, the norm of $x$ is defined to be $|x|^2=\sum_{j=0}^3a_j^2$. Moreover, it is important to note that the norm is multiplicative in $\mathbb H$, namely, for $x, y\in\mathbb H$ we have $|xy|=|x||y|$.

We shall assume that $\mathbb H$ is equipped with an inner product $\langle \cdot,\cdot\rangle$.
The  set $\psi=\{\psi_0, \psi_1,\psi_2,\psi_3\}\subset \mathbb H$ is called a {\em structural set} if
\begin{equation}\label{(1)}
\langle  \psi_k, \psi_m\rangle =\delta_{k,m},
\end{equation}
for $k,m=0,1,2,3$, where $\delta_{k,m}$ denotes the Kronecker's symbol. A structural set $\psi$ is  a basis for $\mathbb H$ and so for any $x \in\mathbb H$ there exist
$x_0, x_1,x_2,x_3\in\mathbb R$ such that $x$, in terms of the basis $\psi$, can be written as
\begin{equation}\label{psix}
x_{\psi} := \sum_{k=0}^3 x_k\psi_k.
\end{equation}

From \eqref{(1)} we deduce that $\langle x,y \rangle_{\psi}=\sum_{k=0}^3 x_k y_k,$ where $x_{\psi} =\sum_{k=0}^3 x_k \psi_k $ and $ y_{\psi} =\sum_{k=0}^3 y_k \psi_k$ with $x_k, y_k\in \mathbb R$ for all $k$.

In the sequel, we identify $\mathbb H$ with $\mathbb R^4$ and we shall consider functions $f$ defined in a bounded domain $\Omega\subset\mathbb R^4 \cong \mathbb H$ with values in $\mathbb H$. The values are written as $f=\sum_{k=0}^3 f_k \psi_k$, where $f_k$ are $\mathbb R$-valued functions defined in $\Omega$. Properties as continuity, differentiability, integrability and so on, which as ascribed to $f$ are referred to the components $f_k$, for $k= 0,1,2,3$. We will follow standard notations, for example $C^{1}(\Omega, \mathbb H)$ denotes the set of continuously differentiable $\mathbb H$-valued functions defined in $\Omega$.

We now recall the concept of $\psi-$hyperholomorphic functions and to this end, we introduce the so-called $\psi$-Fueter operators.
\begin{definition}
The left- and the right-$\psi$-Fueter operators acting on $C^1(\Omega,\mathbb H)$ are defined by
$${}^{{\psi}}\mathcal D[f] :=  \sum_{k=0}^3 \psi_k \partial_k f,$$
and
$$\mathcal  D^{{\psi}}[f]:= \sum_{k=0}^3 \partial_k f \psi_k,$$
respectively, for all $f \in C^1(\Omega,\mathbb H)$, where $\partial_k:=\displaystyle\frac{\partial }{\partial x_k}$.
\end{definition}
It is immediate that when $\psi$ is the standard basis $\psi_{std}:=\{e_0=1, e_1, e_2, e_3\}$  of $\mathbb H$, the two $\psi$-operators coincide with the classical Cauchy-Fueter operators. The classical properties of Cauchy-Fueter regular functions can be extended to $\psi$-regular functions, see \cite{shapiro2,shapiro1}. In particular, we recall that (using the notation in \eqref{psix}) there is a version of the Borel-Pompeiu formula, which is based on the $\psi$-hyperholomorphic  Cauchy kernel
 \[ K_{\psi}(\tau- x )=\frac{1}{2\pi^2} \frac{ \overline{\tau_{\psi} - x_{\psi}}}{|\tau_{\psi} - x_{\psi}|^4}.\]
Let $\partial \Omega$ be a 3-dimensional smooth surface then the quaternionic Borel-Pompieu formula is given by
\begin{align}\label{BPForm}  &  \int_{\partial \Omega }  ( K_{\psi}(\tau-x)\sigma_{\tau}^{\psi} f(\tau)  +  g(\tau)   \sigma_{\tau}^{\psi} K_{\psi}(\tau-x) ) \nonumber  \nonumber  \\
&  - \int_{\Omega} (K_{\psi} (y-x) {}^{\psi}\mathcal D [f] (y)  +  \mathcal D^{{\psi}} [g] (y)   K_{\psi} (y-x))dy   \nonumber \\
		=  & \left\{ \begin{array}{ll}  f(x) + g(x) , &  x\in \Omega,  \\ 0 , &  x\in \mathbb H\setminus\overline{\Omega}.
\end{array} \right.
\end{align}
for all $f,g \in C^1(\overline{\Omega}, \mathbb H).$

It is also useful to recall the differential and integral versions of Stokes formula for the $\psi$-hyperholomorphic functions theory, namely
\begin{align}\label{Stokes} d(g\sigma^{{\psi} }_x f) = & \left( g \ {}^{{\psi}}\mathcal  D[f]+ \  \mathcal D^{{\psi}}[g] f\right)dx, \nonumber \\
\int_{\partial \Omega} g\sigma^\psi_x f =  & \int_{\Omega } \left( g {}^\psi \mathcal  D[f] + \mathcal  D^{{\psi}}[g] f \right)dx,
\end{align}
where $d$ is the exterior differentiation operator, $dx$ is the differential form of the $4-$dimensional volume in $\mathbb R^4$; moreover
$$\sigma^{{\psi} }_{x}:=-{\rm sgn}(\psi) \left( \sum_{k=0}^3 (-1)^k \psi_k d\hat{x}_k\right)$$ is the quaternionic differential form of the $3-$dimensional volume in $\mathbb R^4$ related with $\psi$, where $d\hat{x}_i = dx_0 \wedge dx_1\wedge dx_2  \wedge  dx_3 $ means that the factor $dx_i$ is omitted and ${\rm sgn}(\psi)$ is defined to be $1$, or $-1$ according to the fact that $\psi$ and $\psi_{std}=\{e_0, e_1, e_2, e_3\}$ have the same orientation or not, respectively.

Note that, $|\sigma^{{\psi} }_{x}| = ds_3$ is the differential form of the 3-dimensional volume in $\mathbb R^4.$ We shall write $\sigma_x$ instead of $\sigma^{{\psi_{std}} }_{x}$, for short.

\subsection{Basics on the $(q,q')-$calculus}

We now review briefly basic concepts of the $(q,q')-$calculus. We begin with some notations and terminology and we refer the reader to \cite{NSP} for more information.

For $0 < q' < q \leq 1$, the $(q, q')-$number is defined by
$$[n]_{q,q'} =\displaystyle\frac{q^n - \{q'\}^n}{q - q'}= q^{n-1} + q^{n-2}q'  + \cdots +q \{q'\}^{n-2} + \{q'\}^{n-1}.$$
{It} is a natural generalization of the $q-$number (also called $q$-analogue of $n$) {which is defined as}
$$[n]_{q,1}:= [n]_{q}  =\displaystyle\frac{1 - q^n}{1 - q}= q^{n-1} + q^{n-2}  + \cdots +q  + 1,$$
for any number $n\in \mathbb N\cup\{0\}$. Note that $[n]_{q,q'}=[n]_{q',q}$.
\begin{definition}
The $(q,q')-$derivative of the function $f(x)$ is defined by
\begin{align}\label{$(q,q')-$derivative}
D_{q,q'}f(x) := \frac{f(qx) - f(q'x)}{(q - q')x}, \quad  x\neq 0
\end{align}
and $D_{q,q'}f(0)=f'(0)$, provided that $f$ is differentiable at 0. In addition, if $D_{q, q'}f$  is defined on a neighborhood of $0$ then we set
$${(D_{q, q'}f) (0)}:= \lim_{x\to 0} D_{q, q'}f (x),$$ if it exists.
\end{definition}

\begin{remark} We note that
\begin{align*}
\lim_{x\to 0} D_{q,q'}f (x) = &\lim_{x\to 0 } \frac{f(qx) - f(q'x)}{qx-q'x} =\\
= &\lim_{x\to 0 } \frac{ q}{ q-q'}\frac{f(qx) -f(0) }{qx} - \lim_{x\to 0} \frac{q'}{q-q'} \frac{f(0) - f(q'x) }{-q'x}\\
= & \frac{q}{q-q'}f'(0)- \frac{q'}{q-q'} f'(0)= f'(0).
\end{align*}
{This} is why in general  if $f$ and $D_{q,q'}f$  are defined on a neighborhood of $0$ we set
${(D_{q, q'}f) (0)}:= \lim_{x\to 0} D_{q, q'}f (x),$ if it exists.  A similar definition can be given for $D_{q, q',r}f(0)$ when $f$
 is quaternion-differentiable on the left at $0$.
\end{remark}
As with the $q-$derivative and the ordinary derivative, the action of applying the $(q, q')-$derivative of a function is a linear
operator.

\subsection{On a quaternionic $(q,q')-$derivative}

{Let $0<q'<q\leq 1$.} We now introduce a natural quaternionic generalization of the $(q,q')-$derivative operator (\ref{$(q,q')-$derivative}) acting on  quaternionic valued functions defined over domains of $\mathbb R^4$.
\begin{definition}
Let $\Omega\subset \mathbb H$ be a domain. We will denote by $\mathcal A_{q,q'}(\Omega)$ the set of functions $f\in C^1(\Omega, \mathbb H)$ such that
$$(D_{q,q'} f)(x) := \left[f(qx)- f(q'x) \right] [(q-q') x]^{-1}, \quad \forall  x\in \Omega$$
exists.

Let $\mathcal A_{q,q',r}(\Omega)$ denote the set of $f\in C^1(\Omega, \mathbb H)$ such that
$$(D_{q,q',r} f)(x) :=  [ (q-q') x]^{-1} \left[f(qx)- f(q'x) \right], \quad \forall  x\in \Omega$$
exists.
\end{definition}
\begin{remark}
It is immediate that the operators $D_{q,q'}$, $D_{q,q',r}$, are right and left linear, respectively.
\end{remark}

From now on, we consider $0<q'<q\leq 1$ and  $0<\mathfrak{q}'<\mathfrak{q} \leq 1$.
\begin{proposition}\label{pro14} Suppose $f\in \mathcal A_{q,q'}(\Omega)$, $g\in \mathcal A_{\mathfrak{q}, \mathfrak{q}',r}(\Omega)$ and let ${\bf B}, {\bf C} : \Omega\to \mathbb H\cong \mathbb R^4$ be two conservative vector fields such that ${\bf B} f = D_{q,q'} f - {}^{\psi}\mathcal D [f ]$  and $g{\bf C} = D_{\mathfrak{q},\mathfrak{q}',r} g -  \mathcal D^{\psi} [g]$ on $\Omega.$ Then there exist  $\lambda,\eta \in C^1(\Omega, \mathbb R)$  such that
\begin{align*}
{}^{\psi}\mathcal D [f e^{\lambda}] = (D_{q,q'}f) e^{\lambda}, \quad  \textrm{and}\quad
\mathcal D^{\psi} [g e^{\eta}] = (D_{\mathfrak{q},\mathfrak{q}',r} g) e^{\eta},
\end{align*}
on $\Omega$.
\end{proposition}
\begin{proof} As ${\bf B}$ is conservative, there exists $\lambda \in C^1(\Omega, \mathbb R)$ such that ${\bf B} = {}^{\psi}\mathcal D [\lambda]$  on $\Omega$. Then
\begin{align*}
{}^{\psi}\mathcal D [f e^{\lambda}]  = & \left\{  {}^{\psi}\mathcal D [f]  +  {\bf B}  f \right\}  e^{\lambda}\\
 = & \left\{  {}^{\psi}\mathcal D [f]  +  {}^{\psi}\mathcal D [   \lambda ]  f \right\}  e^{\lambda}\\
= & \left\{  {}^{\psi}\mathcal D [f] +  \left(  (D_{q,q'} f) - {}^{\psi}\mathcal D [f ] \right)   \right\}  e^{\lambda}  .
\end{align*}
The second identity has a similar proof.
\end{proof}
From now on, we let $\Lambda\subset\Omega\subset\mathbb H$ be a domain such that $\partial \Lambda \subset\Omega$ is a $3-$dimensional smooth surface.
\begin{proposition} \label{pro15} Let $f\in \mathcal A_{q,q'}(\Omega) $, $g\in \mathcal A_{\mathfrak{q}, \mathfrak{q}',r}(\Omega)$
 and $\lambda, \eta  \in C^1(\Omega, \mathbb R)$ be as in Proposition \ref{pro14}. Then
\begin{equation}\label{diesis}\begin{split} d(g \ \nu^{{ \psi, \eta,\lambda} }_x \  f ) = & \left[  (D_{\mathfrak{q} ,\mathfrak{q}',r} g) f +  g  (D_{q,q'} f)  \right] dx_{\eta, \lambda},\\
		\int_{\partial \Lambda}  g  \ \nu^{{ \psi, \eta,\lambda} }_x \  f =  &  \int_{\Lambda }
\left[  (D_{\mathfrak{q} ,\mathfrak{q}',r} g) f +  g  (D_{q,q'} f)  \right] dx_{\eta, \lambda},
\end{split}
\end{equation}
where  $\nu^{{\psi, \eta,\lambda} }_x : = e^{\eta + \lambda}\sigma^{{\psi} }_x$ and  $ dx_{\eta, \lambda} :=  e^{\eta+ \lambda} dx$. On the other hand,
\begin{equation}\label{2diesis}	\begin{split}  &  \int_{\partial \Lambda } \left(g  (\tau) \sigma_{\tau}^{\psi}  K_{\psi}(\tau-x)  e^{ \eta (\tau)- \eta(x)}  + K_{\psi}(\tau-x)  e^{ \lambda  (\tau)- \lambda(x)}\sigma_{\tau}^{\psi} f  (\tau)  \right) \\
&- \int_{\Lambda} \left((D_{\mathfrak{q},\mathfrak{q}',r} g) (y) K_{\psi} (y-x) e^{ \eta  (y)- \eta(x)}
+  K_{\psi} (y-x) e^{ \lambda  (y)- \lambda(x)} (D_{q,q'}f)(y) \right) dy   \nonumber \\
		=  &  \left\{ \begin{array}{ll}  g(x) +  f(x), & x\in \Lambda,  \\ 0 , &  x\in   \Omega\setminus \overline{\Lambda}.
\end{array} \right.
\end{split}
\end{equation}
\end{proposition}
\begin{proof}
Formulas in \eqref{diesis} follow immediately by using \eqref{Stokes} and  Proposition \ref{pro14}.
To prove \eqref{2diesis} we apply \eqref{BPForm} to the functions  $g e^{\eta} $ and  $fe^{\lambda}$ and use Proposition \ref{pro14}.
\end{proof}
An immediate consequence is the following:
\begin{corollary} \label{cor16}
Let $f\in \mathcal A_{q,q'}(\Omega) $, $g\in \mathcal A_{\mathfrak{q}, \mathfrak{q}',r}(\Omega)$ and $\lambda, \eta  \in C^1(\Omega, \mathbb R)$ be as in Proposition \ref{pro14}. If $D_{\mathfrak{q},\mathfrak{q}', r} g =0 = D_{q,q'} f  $ on $\Omega$ then
\begin{align*}  		
\int_{\partial \Lambda} g \ \nu^{{\psi, \eta,\lambda} }_x \ f =  &  0,
\end{align*}
and
\begin{align*}
&  \int_{\partial \Lambda } \left(g (\tau)  \sigma_{\tau}^{\psi}  K_{\psi}(\tau-x)  e^{ \eta (\tau)- \eta(x)}  +
 K_{\psi}(\tau-x)  e^{ \lambda  (\tau)- \lambda(x)}\sigma_{\tau}^{\psi} f  (\tau)  \right)    \\
=  &   \left\{ \begin{array}{ll}  g(x) +  f(x), &	x\in \Lambda,  \\ 0 , &  x\in \Omega\setminus \overline{\Lambda}.
\end{array} \right.
\end{align*}
\end{corollary}
\begin{remark}
Let $Z(x)= \bar x$ for all $x\in \mathbb H$ be the quaternionic conjugation mapping. If $\Omega$ is $Z-$invariant ($Z(\Omega)=\Omega$) and  $f\in \mathcal A_{q,q'}(\Omega)$, direct computations give
\begin{align*}
\overline{ (D_{q,q',r} ( Z\circ f \circ Z ) ) (x) } = & \overline{[(q-q')x]^{-1} \left[ Z\circ f \circ Z(qx)- Z\circ f \circ Z(q'x) \right] } \\
= & \overline{ \left[ Z\circ f \circ Z(qx)- Z\circ f \circ Z(q' x) \right] } \ \overline{[(q-q')x]^{-1} } \\
= & \left[ f \circ Z(qx)- f \circ Z(q' x) \right] [(q-q')\bar x]^{-1} \\
= & \left[ f (q\bar x)- f (q' \bar x) \right] [(q-q')\bar x]^{-1} \\
= & D_{q,q'} (f) (\bar x) = D_{q,q'} (f) (Z(x));
\end{align*}
thus
$$\overline{(D_{q,q',r}  ( Z\circ f \circ  Z ) ) (x)} = (D_{q,q'} f) (Z(x) ),$$
for all $x\in \Omega$.
\end{remark}

\section{The $(q,q')-$deformed quaternionic analysis}
Let us start by considering the $\psi-$Fueter operator defined earlier and replace the continuous partial derivatives by certain $(q,q')-$deformed partial derivatives. This process produces a $(q,q')-$deformed $\psi-$Fueter operator which, motivated by \cite{Dem}, can be defined as follows:
\begin{definition}
Let $0<q'<q\leq 1$.  For $f\in C^1(\Omega, \mathbb H)$ we define the $(q,q')-$deformed $\psi-$Fueter operator by
\begin{align*}
({}^{\psi} {\partial}_{{q},{q}'} f )(w,x):= & \sum_{k=0}^3 \psi _k (\partial^{q ,q'}_{x_k} f)(w_0, \dots,   x_k, \dots ,w_3),\  w,x \in \Omega,
\end{align*}
if
\begin{align*}
  (\partial^{q ,q'}_{x_k}f )(w_0, \dots, x_k, \dots ,w_3) := & \frac{f(w_0, \dots, q  x_k, \dots ,w_3) - f(w_0, \dots, q' x_k, \dots ,w_3) }{(q -q') x_k }.
\end{align*}
there exist for  $k=0,1,2,3$, which are called  $(q,q')-$deformed partial derivatives.

Similarly,
\begin{align*}
({}^{\psi} {\partial}_{{q},{q}', r} f )(w,x):= & \sum_{k=0}^3  (\partial^{q ,q'}_{x_k} f)(w_0, \dots,   x_k, \dots ,w_3) \psi _k ,\  w,x \in \Omega.
\end{align*}
 If $x_k=0$ the real components of $(\partial^{q ,q'}_{x_k}f )(w_0, \dots, x_k, \dots ,w_3) $ are defined so, as a consequence, $ (\partial^{q ,q'}_{x_k}f )(w_0, \dots, x_k, \dots ,w_3) $ is defined too.
\end{definition}
Note that if $x_k = 0$ in $(\partial _{\vec q,  \{\vec q\}'} f) (w,x) $  or in  $(\partial _{\vec q,  \{\vec q\}'  , r}  f) (w,x)  $ then the partial derivatives of the real components are already defined at $x_k=0$.\\
Here and in the sequel, we write $(x,w)=(w_0, \dots, x_k, \dots ,w_3)$ instead of $(x)=(x_0, \dots, x_k, \dots ,x_3) $ when all but the variable of interest are held fixed during the differentiation.

\begin{remark} The operator $(\partial^{q ,q'}_{x_k}$ is quaternionic right linear while ${}^{\psi} {\partial}_{{q},{q}', r} $
is quaternionic left linear.
\end{remark}
\begin{remark}\label{Remark12}
In particular, when $w_k=x_k$ for $k=0,1,2,3$ we have
$$(\partial^{q ,q'}_{x_k} f )(w_0, \dots,   x_k, \dots ,w_3)  =  (\partial^{q ,q'}_{x_k}f)(x),$$
so that
$$({}^{\psi} {\partial}_{{ q},{  q}'} f )(x,x)=  \sum_{k=0}^3 \psi _k (\partial^{q ,q'}_{x_k}f)(x), \ \ \textrm{and} \ \
    ({}^{\psi} {\partial}_{{q},{q}', r} f )(x,x) =  \sum_{k=0}^3  (\partial^{q ,q'}_{x_k}f)(x) \psi _k.$$
Moreover, when  $q=1$ we obtain
$$\lim_{q'\to 1^{-}}({}^{\psi} {\partial}_{{1},{q}'}f)(x,x)=\sum_{k=0}^3 \psi _k (\partial _{x_k} f)(x) ={}^{{\psi}}\mathcal D[f](x).$$
$$\lim_{q'\to 1^{-}}({}^{\psi} {\partial}_{{1},{q}', r }f)(x,x)=\sum_{k=0}^3  (\partial _{x_k} f)(x) \psi _k =
\mathcal D  ^ {\psi}  [f](x).$$
Hence, an specific case $({}^{\psi} {\partial}_{{q},{q}'}f)(x,x)$ and $({}^{\psi} {\partial}_{{q},{q}',r}f)(x,x)$ reduces to the standard $\psi-$Fueter operators.
\end{remark}
\begin{remark}
 When $q'=1$, our operator ${}^{\psi} {\partial}_{{q},{q}'}$ is the quaternionic version of the operator studied in \cite{zimmermann2022steps} in the Clifford analysis setting. In particular, ${}^{\psi} {\partial}_{{q},{q}'}$ satisfies the counterpart of axioms (A1) and (A3) where ${}^{\psi} {\partial}_{{q},{q}'}(x)$ and $({}^{\psi} {\partial}_{{q},{q}'})^2$, respectively, are computed. In our case we have:
$$
{}^{\psi} {\partial}_{{q},{q}'}(x)= \sum_{k=0}^3 \psi _k \partial^{q,q'}_{x_k} (x)= -2.
$$
Moreover, one may verify that
$
({}^{\psi} {\partial}_{{q},{q}'})^2f(x)
$
is scalar, in fact $\partial^{q,q'}_{x_k}\partial^{q,q'}_{x_\ell}=\partial^{q,q'}_{x_\ell}\partial^{q,q'}_{x_k}$ and these derivatives have coefficient $\psi_k\psi_\ell+\psi_ell\psi_k=0$ while the second derivatives
$\partial^{q,q'}_{x_k}\partial^{q,q'}_{x_k}$ have a scalar coefficient.
Thus, $({}^{\psi} {\partial}_{{q},{q}'})^2$ is the analogue of the Laplace operator. Finally, we observe that the deformed Fueter variables $(x_k-\psi_k x_0)$, $k=0,\ldots, 3$ are in the kernel of ${}^{\psi} {\partial}_{{q},{q}'}$.
\end{remark}
In the sequel, we shall use the notation ${\vec q}=(q_0,q_1,q_2,q_3)$ and ${\{\vec q}\}'=(q_0',q_1',q_2',q_3')$ such that $0<q_k'<q_k \leq  1$ for all $k=0,1,2,3$.

Let now introduce a multiple, i.e. a multi-variable, $(q,q')-$deformed $\psi-$Fueter operator.
\begin{definition}\label{Def1234}
Let $f\in C^1(\Omega, \mathbb H)$. A multiple $(q,q')-$deformed $\psi-$Fueter operator on $f$, denoted by ${}^{\psi} {\partial}_{{\vec q},{\{\vec q}\}'}f$, if it exists, is defined as follows
\begin{align*}
({}^{\psi} {\partial}_{{\vec q},{\{\vec q}\}'}f)(w,x):= & \sum_{k=0}^3 \psi_k(\partial^{q_k,q_k'}_{x_k} f) (w_0, \dots, x_k,\dots, w_3),
\end{align*}
and a right  multiple $(q,q')-$deformed $\psi-$Fueter operator on $f$,  if it exists, is defined by
\begin{align*}
({}^{\psi} {\partial}_{{\vec q},{\{\vec q}\}' , r }f)(w,x):= & \sum_{k=0}^3(\partial^{q_k,q_k'}_{x_k} f) (w_0, \dots, x_k,\dots, w_3)  \psi_k,
\end{align*}
on $\Omega$.

In particular, when $w=x$ we have
\begin{align*}
(\partial^{q_k,q'_k}_{x_k} f )(w_0, \dots, x_k, \dots ,w_3) = (\partial^{q_k,q'_k}_{x_k} f)(x)
\end{align*}
and
\begin{align*}
({}^{\psi} {\partial}_{{\vec q},{\{\vec q}\}'} f )(x,x) = \sum_{k=0}^3 \psi_k (\partial^{q_k,q_k'}_{x_k} f)(x)=:({}^{\psi} {\partial}_{{\vec q},{\{\vec q}\}'}f)(x)
\end{align*}
on $\Omega$ and the same for ${}^{\psi} {\partial}_{{\vec q},{\{\vec q}\}', r}$.
\end{definition}
\begin{remark}
We note that for $q_k = q$ and $q_k'=q'$ for all $k$ we get
\begin{align*}
(\partial^{q_k,q'_k}_{x_k} f )(w_0, \dots,   x_k, \dots ,w_3) = & (\partial^{q ,q' }_{x_k} f )(w_0, \dots, x_k, \dots ,w_3), \\
 ({}^{\psi} {\partial}_{{\vec q},{\{\vec q}\}'} f )(w,x)= & ({}^{\psi} {\partial}_{ q , q '} f )(w,x) ,\\
 ({}^{\psi} {\partial}_{{\vec q},{\{\vec q}\}' , r } f )(w,x)= & ({}^{\psi} {\partial}_{ q , q ', r } f )(w,x) ,
\end{align*}
for all $x\in \Omega$.
\end{remark}
\begin{definition}
Let us denote by ${\bf A}_{{\vec q},{\{\vec q}\}' }$ (resp. ${\bf A}_{ {\vec q},{\{\vec q}\}',r }$) the set of $f\in C^1(\Omega, \mathbb H)$ such that $$({}^{\psi} {\partial }_{{\vec q},{\{\vec q}\}'} f )(w,x)\qquad \qquad ({\it resp.}({}^{\psi} {\partial }_{{\vec q},{\{\vec q}\}', r}f)(w,x))$$ exists for all $x\in \Omega$.
\end{definition}
\begin{proposition}\label{pro2}  Let $f\in  {\bf A}_{ {\vec q},{\{\vec q}\}'  }$ be such that there exist differentiable real-valued functions  $\lambda_k = \lambda_k(w_0,\dots, x_k, \dots, w_3)$ on $\Omega$ for $k=0,1,2,3$ such that
 \begin{align*}
&  \frac{\partial     \lambda_k(w_0,\dots, x_k, \dots, w_3)   }{\partial x_k} f(w_0,\dots, x_k, \dots, w_3)  \\
= & \left[(\partial^{q_k, q_k'}_{x_k} f )(w_0, \dots,   x_k, \dots ,w_3)  - \frac{\partial    f(w_0,\dots, x_k, \dots, w_3)   }{\partial x_k}  \right]
 , \end{align*}
for all  $x\in \Omega$. Then
\begin{align*}
& {}^{\psi}\mathcal D [ \sum _{k=0}^3 f(w_0,\dots, x_k, \dots, w_3)  e^{\lambda_k (w_0,\dots, x_k, \dots, w_3)}]  \\
=  & ({}^{\psi} {\partial}_{{\vec q},{\{\vec q}\}' } f )(w,x)    \sum _{k=0}^3     e^{\lambda_k (w_0,\dots, x_k, \dots, w_3)}  \\
 &   -
  \sum _{ k,j=0, k\neq j }^3  \psi_k     (\partial^{q_k, q_k'}_{x_k} f )(w_0, \dots,   x_k, \dots ,w_3)    e^{\lambda_j (w_0,\dots, x_j, \dots, w_3)},  \quad \forall  x\in \Omega,
\end{align*}
where the partial derivations in ${}^{\psi}\mathcal D$ are with respect to the real components of $x$ in the basis $\psi$.
\end{proposition}
\begin{proof}
The statement follows from the following computations:
\begin{align*}
&{}^{\psi}\mathcal D [ \sum _{k=0}^3 f(w_0,\dots, x_k, \dots, w_3)  e^{\lambda_k (w_0,\dots, x_k, \dots, w_3)}] \\ = &
\sum _{k=0}^3  \psi_k \frac{\partial    }{\partial x_k}   \left( f(w_0,\dots, x_k, \dots, w_3)  e^{\lambda_k (w_0,\dots, x_k, \dots, w_3)}\right) \\
= &  \sum _{k=0}^3  \psi_k \left[ \frac{\partial   f(w_0,\dots, x_k, \dots, w_3) }{\partial x_k}    +
  \frac{\partial     \lambda_k (w_0,\dots, x_k, \dots, w_3)     }{\partial x_k}    f(w_0,\dots, x_k, \dots, w_3)  \right]\times \\
   & \hspace{2cm }  e^{\lambda_k (w_0,\dots, x_k, \dots, w_3)}\\  	
 = & \sum _{k=0}^3 \psi_k (\partial^{q_k,q_k'}_{x_k} f )(w_0,\dots, x_k, \dots, w_3)e^{\lambda_k (w_0,\dots, x_k, \dots, w_3)}=
\end{align*} 	
\begin{align*}
 =&({}^{\psi} {\partial}_{{\vec q},{\{\vec q}\}'} f )(w,x) \sum _{k=0}^3 e^{\lambda_k (w_0,\dots, x_k, \dots, w_3)} \\
& -  \sum _{ k,j=0, k\neq j }^3 \psi_k (\partial^{q_k,q_k'}_{x_k} f )(w_0,\dots, x_k, \dots, w_3) e^{\lambda_j (w_0,\dots, x_j, \dots, w_3)}.
\end{align*}
\end{proof}
\begin{remark}\label{rem1234} In an analogous way, we can state the counterpart of Proposition \ref{pro2} for the right-$\psi$-Fueter operator. In fact, let $g\in {\bf A}_{ \vec{\mathfrak{q}}, \{\vec {\mathfrak{q}}\}', r}$ and suppose there exist real-valued functions
$\eta_k = \eta_k(w_0,\dots, x_k, \dots, w_3)$ defined on $\Omega$, for $k=0,1,2,3$, such that
\begin{align*}
&  \frac{\partial     \eta_k(w_0,\dots, x_k, \dots, w_3)   }{\partial x_k} g(w_0,\dots, x_k, \dots, w_3)  \\
= & \left[(\partial^{\mathfrak{q}_k, \mathfrak{q}_k'}_{x_k} g )(w_0, \dots,   x_k, \dots ,w_3)  - \frac{\partial    g (w_0,\dots, x_k, \dots, w_3)   }{\partial x_k}  \right],
\end{align*}
for all  $x\in \Omega$. Then, computations similar to those ones in the previous proof show that
\begin{align*}
&  \mathcal D^{\psi} [ \sum _{k=0}^3 g(w_0,\dots, x_k, \dots, w_3)  e^{\eta_k (w_0,\dots, x_k, \dots, w_3)}]  \\
=  & ({}^{\psi} {\partial}_{{\vec{\mathfrak{q}}},{\{\vec{\mathfrak{q}}}\}', r } g )(w,x)    \sum _{k=0}^3
e^{\eta_k (w_0,\dots, x_k, \dots, w_3)}  \\
 &   -
  \sum _{k,j=0,  k\neq j }^3 (\partial^{\mathfrak{q}_k, \mathfrak{q}_k'}_{x_k} g )(w_0, \dots, x_k, \dots ,w_3) e^{\eta_j (w_0,\dots, x_j, \dots, w_3)}  \psi_k,  \quad \forall  x\in \Omega,
\end{align*}
\end{remark}
We will use the symbol ${\bf A}_{ \vec q ,  \{\vec q \}'}(\Omega)$ to denote the subset of ${\bf A}_{ \vec q ,  \{\vec q \}'}$ whose elements satisfy the assumptions of Proposition \ref{pro2}.

Similarly, we denote by ${\bf A}_{{\vec {\mathfrak q} }, {\{\vec {\mathfrak q}}\}',r}(\Omega)$ the set of all functions satisfying the conditions in Remark \ref{rem1234}.

For the computations in the next section, it is useful to introduce some more notations. Given $f \in {\bf A}_{{\vec q}, {\{\vec q}\}'}(\Omega)$ and  $g \in {\bf A}_{{\vec{\mathfrak q}}, {\{\vec {\mathfrak q} }\}', r }(\Omega)$ we define
$$(I_{{\vec q}, {\{\vec q}\}'} f)(w,x):=\sum _{k=0}^3 f(w_0,\dots, x_k, \dots, w_3)  e^{\lambda_k (w_0,\dots, x_k, \dots, w_3)}$$
and
$$(I_{{ \vec{\mathfrak q} }, \{  \vec{\mathfrak q} \}'} g)(w,x):=\sum _{k=0}^3 g(w_0,\dots, x_k, \dots, w_3)  e^{\eta_k (w_0,\dots, x_k, \dots, w_3)}.$$
Proposition \ref{pro2} now shows that
\begin{align*}
 {}^{\psi}\mathcal D [(I_{{\vec q}, {\{\vec q}\}'} f)(w,x)]   = &({}^{\psi} {\partial}_{{\vec q},{\{\vec q}\}'} f )(w,x) \sum _{k=0}^3     e^{\lambda_k (w_0,\dots, x_k, \dots, w_3)} \\
 & -
  \sum _{ k,j=0, k\neq j }^3  \psi_k     (\partial^{q_k,q_k'}_{x_k} f )(w_0,\dots, x_k, \dots, w_3)   e^{\lambda_j (w_0,\dots, x_j, \dots, w_3)}
\end{align*}
and from Remark \ref{rem1234} we see that
\begin{align*}
  \mathcal D ^{\psi} [(I_{{\vec {\mathfrak q}}, {\{\vec {\mathfrak q} }\}'} g)(w,x)] = &({}^{\psi} {\partial}_{{\vec {\mathfrak q}},{\{\vec {\mathfrak q}}\}', r } g )(w,x) \sum _{k=0}^3 e^{\eta_k (w_0,\dots, x_k, \dots, w_3)} \\
 & -
  \sum _{ k,j=0, k\neq j }^3 (\partial^{{\mathfrak q}_k,{\mathfrak q}_k'}_{x_k} g )(w_0,\dots, x_k, \dots, w_3)   e^{\eta_j (w_0,\dots, x_j, \dots, w_3)} \psi_k.
\end{align*}

\section{Main Results}
In this section, we formulate and prove our main results: the Stokes and Borel-Pompeiu type formulas associated with the operators ${}^{\psi} {\partial }_{{\vec q},{\{\vec q}\}'}$ and  ${}^{\psi} {\partial }_{{\vec{\mathfrak q}},{\{\vec{\mathfrak q}}\}', r}$. Moreover, some direct consequences such as the Cauchy formula and the Cauchy theorem are obtained.\\
We begin by proving a technical result, useful in the sequel.
\begin{proposition}\label{prop}
Let $f\in {\bf A}_{{\vec q}, {\{\vec q}\}'}(\Omega)$ and $\lambda_{k}$ as in Proposition \ref{pro2} and let
$g\in {\bf A}_{{\vec {\mathfrak q}}, {\{\vec{\mathfrak q}}\}', r}(\Omega) $ and $\eta_{k}$  as in Remark \ref{rem1234},
for $k=0,1,2,3$. Then
$$d\left((I_{{\vec{\mathfrak q}}, \{\vec {\mathfrak q}\}', r} g)(w,x) \sigma^{{\psi} }_x (I_{{\vec q}, \{\vec q\}'}f)(w,x) \right)$$
$$= ({}^{\psi} {\partial}_{\vec{\mathfrak q}, \{\vec{\mathfrak q}\}', r } g )(w,x)  (I_{{\vec q },  \{\vec q\}'} f)(w,x) \sum _{k=0}^3 e^{\eta_k (w_0,\dots, x_k, \dots, w_3)} dx$$
\begin{align*}
 &  +
  (I_{   \vec{\mathfrak q} ,   \{\vec {\mathfrak q} \}',r } g)(w,x)    ({}^{\psi} {\partial}_{  \vec q ,  \{\vec q\}'} f )(w,x)
   \sum _{k=0}^3 e^{    \lambda_k  (w_0,\dots, x_k, \dots, w_3) }   dx\\
&  -  \sum _{ k,j=0, k\neq j }^3  \psi_k (\partial^{  {\mathfrak q}_k,{\mathfrak q}_k'}_{x_k} g )(w_0,\dots, x_k, \dots, w_3)e^{\eta_j (w_0,\dots, x_j, \dots, w_3)} (I_{  {\vec q  }, \{\vec   q \}'  } f)(w,x)   dx \\
&  - (I_{   \vec{\mathfrak q}   ,   \{\vec {\mathfrak q} \}',r } g)(w,x)  \sum _{ k,j=0, k\neq j }^3  \psi_k (\partial^{q_k,q_k'}_{x_k} f )(w_0,\dots, x_k, \dots, w_3)e^{\lambda_j (w_0,\dots, x_j, \dots, w_3)}dx.
\end{align*}
and
\begin{align*}
& 		\int_{\partial \Lambda}    (I_{  {\vec{\mathfrak q}}, \{\vec {\mathfrak q} \}',r } g)(w,x)   \sigma^{{\psi} }_x  (I_{  {\vec q},   \{\vec q\}' }  f)(w,x)   \\
 = &  \int_{\Lambda}
   ({}^{\psi} {\partial}_{    \vec{\mathfrak q}  ,     \{\vec{\mathfrak q}\}', r } g )(w,x)  (I_  {   {\vec q },  \{\vec  q  \}'  } f)(w,x)
  \sum _{k=0}^3 e^{    \eta_k (w_0,\dots, x_k, \dots, w_3) }   dx    \\
 &  +    \int_{\Lambda}
  (I_{   \vec{\mathfrak q} ,   \{\vec {\mathfrak q} \}',r } g)(w,x)    ({}^{\psi} {\partial}_{  \vec q ,  \{\vec q\}'} f )(w,x)
   \sum _{k=0}^3 e^{    \lambda_k  (w_0,\dots, x_k, \dots, w_3) }   dx\\
&  -   \int_{\Lambda}   \sum _{ k,j=0, k\neq j }^3  \psi_k (\partial^{  {\mathfrak q}_k,{\mathfrak q}_k'}_{x_k} g )(w_0,\dots, x_k, \dots, w_3)e^{\eta_j (w_0,\dots, x_j, \dots, w_3)} (I_{  {\vec q  }, \{\vec   q \}'  } f)(w,x)   dx \\
&  -    \int_{\Lambda}    (I_{   \vec{\mathfrak q}   ,   \{\vec {\mathfrak q} \}',r } g)(w,x)  \sum _{ k,j=0, k\neq j }^3  \psi_k (\partial^{q_k,q_k'}_{x_k} f )(w_0,\dots, x_k, \dots, w_3)e^{\lambda_j (w_0,\dots, x_j, \dots, w_3)}dx.
\end{align*}
Furthermore,
$$\int_{\partial \Lambda }
  (I_{  \vec{\mathfrak q} , \{\vec {\mathfrak q} \}' ,r} g)(w,\tau)  \sigma_{\tau}^{\psi} K_{\psi}(\tau-x)
+
K_{\psi}(\tau-x)\sigma_{\tau}^{\psi} (I_{{\vec q}, {\{\vec q\}'}} f)(w,\tau) $$
\begin{align*}
&  - \int_{\Lambda}  ({}^{\psi} {\partial}_{ \vec{\mathfrak q},  \{\vec {\mathfrak q}\}', r } g )(w,y)   K_{\psi} (y-x) [ \sum _{k=0}^3     e^{\eta_k (w_0,\dots, y_k, \dots, w_3)}] dy \\
&  - \int_{\Lambda}  K_{\psi} (y-x) [ \sum _{k=0}^3     e^{\lambda_k (w_0,\dots, y_k, \dots, w_3)}] ({}^{\psi} {\partial}_{{\vec q}, {\{\vec q\}'}} f )(w,y) dy \\
& +  \int_{\Lambda}   \sum _{ k,j=0, k\neq j }^3  (\partial^{
{\mathfrak q}_k,{\mathfrak q}_k'}_{y_k} g )(w_0,\dots, y_k, \dots, w_3) \psi_k e^{\eta_j (w_0,\dots, y_j, \dots, w_3)}  K_{\psi} (y-x) dy  \nonumber \\
& +  \int_{\Lambda}  K_{\psi} (y-x) \sum _{ k,j=0, k\neq j }^3 \psi_k (\partial^{q_k,q_k'}_{y_k} f )(w_0,\dots, y_k, \dots, w_3) e^{\lambda_j (w_0,\dots, y_j, \dots, w_3)} dy  \nonumber \\
=  & \left\{ \begin{array}{ll}   \displaystyle  \sum _{k=0}^3 (f  e^{\lambda_k } + g   e^{\eta_k })(w_0,\dots, x_k, \dots, w_3)
   , &  x\in \Lambda,  \\ 0 , &  x\in   \Omega\setminus \overline{\Lambda}.
\end{array} \right.
\end{align*}
\end{proposition}
\begin{proof}
Let us apply \eqref{Stokes} to the functions
\begin{align}\label{funct}
(I_{{\vec q}, {\{\vec q}\}'} f)(w,x): x\mapsto \sum _{k=0}^3 f(w_0,\dots, x_k, \dots, w_3) e^{\lambda_k (w_0,\dots, x_k, \dots, w_3)} \nonumber \\
(I_{    \vec {\mathfrak q},  \{\vec {\mathfrak q}\}'} g)(w,x): x\mapsto \sum _{k=0}^3 g(w_0,\dots, x_k, \dots, w_3) e^{\eta_k (w_0,\dots, x_k, \dots, w_3)} .
\end{align}
By Proposition \ref{pro2} we obtain the first two identities:
\begin{align*}
& d\left((I_{\vec{\mathfrak q}, \{\vec {\mathfrak q} \}'}g)(w,x)\sigma^{{\psi} }_x (I_{\vec q, \{\vec q\}'} f)(w,x) \right) \\
=&
\left\{\mathcal D^{\psi}\left[(I_{\vec {\mathfrak q}, \{\vec {\mathfrak q} \}'} g)(w,x)\right] (I_{\vec q, \{\vec q \}'} f)(w,x) +
(I_{\vec {\mathfrak q}, \{\vec {\mathfrak q}\}'}g)(w,x) {}^{{\psi}}\mathcal D\left[(I_{\vec q, \{\vec q\}'}f)(w,x)\right] \right\}dx
\end{align*}
and
\begin{align*}
&		\int_{\partial \Lambda}    (I_{    \vec {\mathfrak q} ,  \{\vec {\mathfrak q} \}' }    g)(w,x)  \sigma^{{\psi} }_x (I_{    \vec q ,  \{\vec q\}' }    f)(w,x)  \\
 =  &  \int_{\Omega} \left\{      \mathcal D^{\psi}\left[(I_{    \vec {\mathfrak q}  ,  \{\vec {\mathfrak q} \}', }  g)  (w,x)\right]    (I_{    \vec q ,  \{\vec q \}' }  f)(w,x)   +
(I_{    \vec {\mathfrak q} ,  \{\vec {\mathfrak q} \}' }  g)(w,x)   {}^{{\psi}}\mathcal D\left[(I_{    \vec q  ,  \{\vec q\}' }  f)  (w,x)\right] \right\}dx.
\end{align*}
On the other hand, from \eqref{BPForm} we have
\begin{align*}
&  \int_{\partial \Lambda } K_{\psi}(\tau-x)\sigma_{\tau}^{\psi} (I_{{\vec q}, {\{\vec q\}'}} f)(w,\tau)
  - \int_{\Lambda}  K_{\psi} (y-x){}^{\psi}\mathcal D \left[(I_{{\vec q}, {\{\vec q\}'} }f)(w,y)\right]dy \nonumber \\
& = \left\{\begin{array}{ll} (I_{{\vec q}, {\{\vec q\}'}}f)(w,x), & x\in \Lambda, \\ 0, & x\in \Omega\setminus \overline{\Lambda},
\end{array} \right.
\end{align*}
Similarly, by Proposition \ref{prop} we obtain
\begin{align*}
&  \int_{\partial \Lambda } K_{\psi}(\tau-x)\sigma_{\tau}^{\psi} (I_{{\vec q}, {\{\vec q\}'}} f)(w,\tau)
  - \int_{\Lambda}  K_{\psi} (y-x) [ \sum _{k=0}^3     e^{\lambda_k (w_0,\dots, y_k, \dots, w_3)}] ({}^{\psi} {\partial}_{{\vec q},{\{\vec q\}'}} f )(w,y) dy  \\
& +  \int_{\Lambda}  K_{\psi} (y-x) \sum _{ k,j=0, k\neq j }^3  \psi_k (\partial^{q_k,q_k'}_{y_k} f ) (w_0,\dots, y_k, \dots, w_3) e^{\lambda_j (w_0,\dots, y_j, \dots, w_3)} dy   \nonumber \\
=  &       \left\{ \begin{array}{ll}  \displaystyle  \sum _{k=0}^3 f(w_0,\dots, x_k, \dots, w_3)  e^{\lambda_k (w_0,\dots, x_k, \dots, w_3)},
&  x\in\Lambda,  \\ 0 , &  x\in   \Omega\setminus\overline{\Lambda} .
\end{array} \right.
\end{align*}
The formula for $g$ is obtained with similar calculations.
\end{proof}
\begin{remark}
Taking $x=w$ in the last formula of Proposition \ref{prop}, we obtain that
\begin{align*}
& \ \quad  \int_{\partial \Lambda }
  (I_{  \vec{\mathfrak q} , \{\vec {\mathfrak q} \}' ,r} g)(w,\tau)  \sigma_{\tau}^{\psi} K_{\psi}(\tau-w)
+
K_{\psi}(\tau-w)\sigma_{\tau}^{\psi} (I_{{\vec q}, {\{\vec q\}'}} f)(w,\tau) \\
&  - \int_{\Lambda}  ({}^{\psi} {\partial}_{ \vec{\mathfrak q},  \{\vec {\mathfrak q}\}', r } g )(w,y)   K_{\psi} (y-w) [ \sum _{k=0}^3     e^{\eta_k (w_0,\dots, y_k, \dots, w_3)}] dy \\
&  - \int_{\Lambda}  K_{\psi} (y-w) [ \sum _{k=0}^3     e^{\lambda_k (w_0,\dots, y_k, \dots, w_3)}] ({}^{\psi} {\partial}_{{\vec q}, {\{\vec q\}'}} f )(w,y) dy \\
& +  \int_{\Lambda}   \sum _{ k,j=0, k\neq j }^3  (\partial^{
{\mathfrak q}_k,{\mathfrak q}_k'}_{y_k} g )(w_0,\dots, y_k, \dots, w_3) \psi_k e^{\eta_j (w_0,\dots, y_j, \dots, w_3)}  K_{\psi} (y-w) dy
\end{align*}
\begin{align*}
& +  \int_{\Lambda}  K_{\psi} (y-w) \sum _{ k,j=0, k\neq j }^3 \psi_k (\partial^{q_k,q_k'}_{y_k} f )(w_0,\dots, y_k, \dots, w_3) e^{\lambda_j (w_0,\dots, y_j, \dots, w_3)} dy  \nonumber \\
=  & \left\{ \begin{array}{ll}   \displaystyle   f(w) \sum _{k=0}^3  e^{\lambda_k (w_0,\dots, w_k, \dots, w_3) } + g (w)    \sum _{k=0}^3  e^{\eta_k (w_0,\dots, w_k, \dots, w_3)})
   , &  w\in \Lambda,  \\ 0 , &  w\in   \Omega\setminus \overline{\Lambda}.
\end{array} \right.
\end{align*}
\end{remark}
An immediate consequence of this discussion is the following:
\begin{corollary} \label{coro} Under the same hypothesis and notations of Proposition \ref{prop} if, in addition, it holds that
$$({}^{\psi}\partial_{\vec{\mathfrak  q},  \{\vec {\mathfrak  q} \}', r}  g)(w,x)= 0 =({}^{\psi}\partial_{{\vec q},{\{\vec q}\}'}f)(w,x),\quad for all x\in \Omega,$$
then
\begin{align*}
& d\left(    (I_{  {\vec{\mathfrak q}}, \{\vec {\mathfrak q} \}',r } g)(w,x)   \sigma^{{\psi} }_x  (I_{  {\vec q},   \{\vec q\}' }  f)(w,x) \right)  \\
= &  -  \sum _{ k,j=0, k\neq j }^3  \psi_k (\partial^{  {\mathfrak q}_k,{\mathfrak q}_k'}_{x_k} g )(w_0,\dots, x_k, \dots, w_3)e^{\eta_j (w_0,\dots, x_j, \dots, w_3)} (I_{{\vec q }, \{\vec q \}'} f)(w,x)   dx \\
&  - (I_{   \vec{\mathfrak q},   \{\vec {\mathfrak q} \}',r } g)(w,x)  \sum _{ k,j=0, k\neq j }^3\psi_k (\partial^{q_k,q_k'}_{x_k} f )(w_0,\dots, x_k, \dots, w_3)e^{\lambda_j (w_0,\dots, x_j, \dots, w_3)}dx.
\end{align*}
Moreover
\begin{align*}
& 		\int_{\partial \Lambda} (I_{{\vec{\mathfrak q}}, \{\vec {\mathfrak q} \}',r } g)(w,x)   \sigma^{{\psi} }_x  (I_{  {\vec q},   \{\vec q\}' }  f)(w,x)   \\
 =  & -   \int_{\Lambda}   \sum _{ k,j=0, k\neq j }^3  \psi_k (\partial^{  {\mathfrak q}_k,{\mathfrak q}_k'}_{x_k} g )(w_0,\dots, x_k, \dots, w_3)e^{\eta_j (w_0,\dots, x_j, \dots, w_3)} (I_{{\vec q}, \{\vec q\}'}f)(w,x)dx,
\end{align*}
\begin{align*}
&  -    \int_{\Lambda} (I_{   \vec{\mathfrak q},   \{\vec {\mathfrak q} \}',r } g)(w,x)  \sum _{ k,j=0, k\neq j }^3  \psi_k (\partial^{q_k,q_k'}_{x_k} f )(w_0,\dots, x_k, \dots, w_3)e^{\lambda_j (w_0,\dots, x_j, \dots, w_3)}dx
\end{align*}
and
\begin{align*}
& \ \quad  \int_{\partial \Lambda } (I_{\vec{\mathfrak q}, \{\vec {\mathfrak q} \}' ,r} g)(w,\tau)  \sigma_{\tau}^{\psi} K_{\psi}(\tau-x)
+ K_{\psi}(\tau-x)\sigma_{\tau}^{\psi} (I_{{\vec q}, {\{\vec q\}'}} f)(w,\tau) \\
& +  \int_{\Lambda}   \sum _{ k,j=0, k\neq j }^3  (\partial^{
{\mathfrak q}_k,{\mathfrak q}_k'}_{y_k} g )(w_0,\dots, y_k, \dots, w_3) \psi_k e^{\eta_j (w_0,\dots, y_j, \dots, w_3)}  K_{\psi} (y-x) dy  \nonumber \\
& +  \int_{\Lambda}  K_{\psi} (y-x) \sum _{ k,j=0, k\neq j }^3 \psi_k (\partial^{q_k,q_k'}_{y_k} f )(w_0,\dots, y_k, \dots, w_3) e^{\lambda_j (w_0,\dots, y_j, \dots, w_3)} dy  \nonumber \\
=  & \left\{ \begin{array}{ll}   \displaystyle  \sum _{k=0}^3 (f  e^{\lambda_k } + g   e^{\eta_k })(w_0,\dots, x_k, \dots, w_3)
   , &  x\in \Lambda,  \\ 0 , &  x\in   \Omega\setminus \overline{\Lambda}.
\end{array} \right.
\end{align*}
\end{corollary}
\begin{remark}
Taking $x=w$ in the last formula of Corollary \ref{coro} yields
\begin{align*}
& \ \quad  \int_{\partial \Lambda}(I_{\vec{\mathfrak q}, \{\vec {\mathfrak q} \}', r}g)(w,\tau) \sigma_{\tau}^{\psi} K_{\psi}(\tau-w)
+
K_{\psi}(\tau-w)\sigma_{\tau}^{\psi} (I_{{\vec q}, {\{\vec q\}'}} f)(w,\tau) \\
& +  \int_{\Lambda}   \sum _{ k,j=0, k\neq j }^3  (\partial^{
{\mathfrak q}_k,{\mathfrak q}_k'}_{y_k} g )(w_0,\dots, y_k, \dots, w_3) \psi_k e^{\eta_j (w_0,\dots, y_j, \dots, w_3)}  K_{\psi} (y-w) dy  \nonumber \\
& +  \int_{\Lambda}  K_{\psi} (y-w) \sum _{ k,j=0, k\neq j }^3 \psi_k (\partial^{q_k,q_k'}_{y_k} f )(w_0,\dots, y_k, \dots, w_3) e^{\lambda_j (w_0,\dots, y_j, \dots, w_3)} dy  \nonumber \\
\end{align*}
\begin{align*}
=  & \left\{ \begin{array}{ll}   \displaystyle   f(w) \sum _{k=0}^3 e^{\lambda_k  (w_0,\dots, w_k, \dots, w_3)} + g (w)  \sum _{k=0}^3 e^{\eta_k  (w_0,\dots, w_k, \dots, w_3)}
   , &  w\in \Lambda,  \\ 0 , &  w\in   \Omega\setminus \overline{\Lambda}.
\end{array} \right.
\end{align*}
\end{remark}
Now, at the price of weakening the hypothesis of Proposition \ref{pro2}, the next statement provides a more general result.
\begin{proposition}\label{pro3} Let $f\in {\bf A}_{ \vec q ,  \{\vec q \}'}$. Suppose that there exists some real-valued functions $\lambda_{j,k}(x_k) = \lambda_{j,k}(w_0,\dots, x_k, \dots, w_3)$  for $j,k=0,1,2,3$  such that
\begin{align*}
&\partial_k \lambda_{j,k} (w_0,\dots, x_k, \dots, w_3) f_j(w_0,\dots, x_k, \dots, w_3) \\
= & \left[ (\partial^{q_k,q_k'}_{x_k} f_j )(w_0,\dots, x_k, \dots, w_3) - \partial_k f_j(w_0,\dots, x_k, \dots, w_3)\right]
\end{align*}
for all $x\in \Omega.$ Then
\begin{align*}
& {}^{\psi}\mathcal D [ \sum _{j,k=0}^3\psi_j f_j(w_0,\dots, x_k, \dots, w_3)  e^{\lambda_{j,k} (w_0,\dots, x_k, \dots, w_3)}]  \\
= &  ({}^{\psi} {\partial}_{ \vec q , \{\vec q \}'} f )(w,x)\sum _{\ell,m=0}^3 e^{\lambda_{\ell,m} (w_0,\dots, x_m, \dots, w_3)} \\
&  - \sum _{ {{   \begin{array}{c}   j,k=0, j\not=k  \\  \ell, m =0, \ell\not= m \end{array} }}}^3 \psi_k \psi_j (\partial^{q_k,q_k'}_{x_k} f_j )(w_0,\dots, x_k, \dots, w_3) e^{\lambda_{\ell,m} (w_0,\dots, x_m, \dots, w_3)}.
\end{align*}
\end{proposition}

\begin{proof} For fixed $j\in\{0,1,2,3\}$, direct computations show that
\begin{align*}
& {}^{\psi}\mathcal D [ \sum _{k=0}^3 f_j(w_0,\dots, x_k, \dots, w_3)  e^{\lambda_{j,k} (w_0,\dots, x_k, \dots, w_3)}]    \\
 = & \sum _{k=0}^3 \psi_k(\partial^{q_k,q_k'}_{x_k} f_j )(w_0,\dots, x_k, \dots, w_3) e^{\lambda_{j,k} (w_0,\dots, x_k, \dots, w_3)}.
\end{align*}
Therefore,
\begin{align*}
 & {}^{\psi}\mathcal D [ \sum _{j,k=0}^3\psi_j f_j(w_0,\dots, x_k, \dots, w_3)  e^{\lambda_{j,k} (w_0,\dots, x_k, \dots, w_3)}]\\
   = &   \sum _{j,k=0}^3  \psi_k \psi_j (\partial^{q_k,q_k'}_{x_k} f_j )(w_0,\dots, x_k, \dots, w_3) e^{\lambda_{j,k} (w_0,\dots, x_k, \dots, w_3)} \\
	 = & ({}^{\psi} {\partial}_{    \vec q   , \{\vec q \}'} f )(w,x)     \sum _{\ell,m=0}^3     e^{\lambda_{\ell,m} (w_0,\dots, x_m, \dots, w_3)}
\end{align*}
\begin{align*}
 & - \sum _{ {{   \begin{array}{c}   j,k=0, j\not=k  \\  \ell, m =0, \ell\not= m \end{array} }}}^3 \psi_k \psi_j (\partial^{q_k,q_k'}_{x_k} f_j )(w_0,\dots, x_k, \dots, w_3) e^{\lambda_{\ell,m} (w_0,\dots, x_m, \dots, w_3)}.
\end{align*}
\end{proof}
\begin{remark}\label{rem3.5} Let   $g\in {\bf A}_{ \vec {\mathfrak q} ,  \{\vec {\mathfrak q} \}', r }$ be such that there exists  real-valued functions $\eta_{j,k}(x_k) = \eta_{j,k}(w_0,\dots, x_k, \dots, w_3)$  for $j,k=0,1,2,3$  for which
\begin{align*}
&\partial_k \lambda_{j,k} (w_0,\dots, x_k, \dots, w_3) g_j(w_0,\dots, x_k, \dots, w_3) \\
= & \left[ (\partial^{q_k,q_k'}_{x_k} g_j )(w_0,\dots, x_k, \dots, w_3) - \partial_k g_j(w_0,\dots, x_k, \dots, w_3)\right]
\end{align*}
for all $x\in \Omega.$ Then from  similar way to the previous proposition we can see that
\begin{align*}
&  \mathcal D^{\psi} [ \sum _{j,k=0}^3\psi_j g_j(w_0,\dots, x_k, \dots, w_3)  e^{\eta_{j,k} (w_0,\dots, x_k, \dots, w_3)}]  \\
= &  ({}^{\psi} {\partial}_{ \vec {\mathfrak q} , \{\vec {\mathfrak q} \}', r } g )(w,x)\sum _{\ell,m=0}^3 e^{\eta_{\ell,m} (w_0,\dots, x_m, \dots, w_3)} \\
&  - \sum _{ {{   \begin{array}{c}   j,k=0, j\not=k  \\  \ell, m =0, \ell\not= m \end{array} }}}^3 \psi_j  \psi_k (\partial^{  {\mathfrak q}_k,{\mathfrak q}_k'}_{x_k} g_j )(w_0,\dots, x_k, \dots, w_3) e^{\eta_{\ell,m} (w_0,\dots, x_m, \dots, w_3)}.
\end{align*}
\end{remark}
\begin{remark}
Note that if $f_j(w_0,\dots, x_k, \dots, w_3) \neq 0 $ for all $x\in \Omega$ and all $j,k=0,1,2,3$ it follows
\begin{align*}
& \partial_k \lambda_{j,k} (w_0,\dots, x_k, \dots, w_3)=\\
& \frac{(\partial^{q_k,q_k'}_{x_k} f_j )(w_0,\dots, x_k, \dots, w_3) - \partial_k f_j(w_0,\dots, x_k, \dots, w_3)}{f_j(w_0,\dots, x_k, \dots, w_3)}
\end{align*}
for all $x\in \Omega$. Furthermore, $\lambda_{j,k}$ for all $j,k=0,1,2,3$ can be represent by real dimensional integrals. The same fact for $g$.
\end{remark}
Let us denote by $\mathcal A_{ \vec q , \{\vec q \}'}(\Omega)$ the collection of all $f\in {\bf A}_{ \vec q , \{\vec q \}'}$ that satisfy the hypothesis of Proposition \ref{pro3}. Similarly, the set $\mathcal A_{ \vec {\mathfrak q} , \{\vec {\mathfrak q} \}', r }(\Omega)$ consists of
$g \in {\bf A}_{ \vec {\mathfrak q} , \{\vec {\mathfrak q} \}', r} $ satisfying the conditions in Remark \ref{rem3.5}.

Given $f \in \mathcal A_{\vec q, \{\vec q\}'}(\Omega)$ and $g\in \mathcal A_{\vec{\mathfrak q}, \{\vec{\mathfrak q}\}', r}(\Omega)$ we define
\begin{align*}
(\mathcal I_{\vec q, \{\vec q \}'}f)(w,x) :=& \sum _{j,k=0}^3\psi_j f_j(w_0,\dots, x_k, \dots, w_3) e^{\lambda_{j,k}(w_0,\dots, x_k, \dots, w_3)},
\\
(\mathcal I_{ \vec{ \mathfrak q}, \{\vec{ \mathfrak q} \}'}  g )(w,x) :=&  \sum _{j,k=0}^3\psi_j g_j(w_0,\dots, x_k, \dots, w_3)  e^{\eta_{j,k} (w_0,\dots, x_k, \dots, w_3)},
\end{align*}
for all $x\in \Omega$. Then  Proposition \ref{pro3} and Remark \ref{rem3.5} show  that
\begin{align*}
& {}^{\psi}\mathcal D [ (\mathcal I_{ \vec q,  \{\vec q\}'} f)(w,x)] =  ({}^{\psi} {\partial}_{\vec q, \{\vec q\}'} f )(w,x)
 \sum _{\ell,m=0}^3 e^{\lambda_{\ell,m} (w_0,\dots, x_m, \dots, w_3)} \\
& \qquad - \sum _{ {{   \begin{array}{c}   j,k=0, j\not=k  \\  \ell, m =0, \ell\not= m \end{array} }}}^3 \psi_k \psi_j (\partial^{q_k,q_k'}_{x_k}f_j )(w_0,\dots, x_k, \dots, w_3)e^{\lambda_{\ell,m} (w_0,\dots, x_m, \dots, w_3)},
\end{align*}
and
\begin{align*}
& \mathcal D^{\psi} [ (\mathcal I_{\vec { \mathfrak q}  ,\{\vec { \mathfrak q} \}'} g)(w,x)] =   ({}^{\psi} {\partial}_{\vec { \mathfrak q},  \{\vec { \mathfrak q}\}', r } g )(w,x)
 \sum _{\ell,m=0}^3 e^{\eta_{\ell,m} (w_0,\dots, x_m, \dots, w_3)} \\
&\qquad  - \sum _{ {{   \begin{array}{c}   j,k=0, j\not=k  \\  \ell, m =0, \ell\not= m \end{array} }}}^3 \psi_j \psi_k (\partial^{{ \mathfrak q}_k,{ \mathfrak q}_k'}_{x_k}g_j )(w_0,\dots, x_k, \dots, w_3)e^{\eta_{\ell,m} (w_0,\dots, x_m, \dots, w_3)},
\end{align*}
for all $x\in \Omega$.
\begin{proposition} \label{proposition}
Let $f\in  \mathcal A_{{\vec q},{\{\vec q}\}'}(\Omega)$,
$g\in  \mathcal A_{ \vec {\mathfrak q}, \{\vec {\mathfrak q} \}'}(\Omega)$,
and let $\lambda_{j,k}$,  $\eta_{j,k}$ for $j,k=0,1,2,3$  as in Proposition \ref{pro3} and Remark \ref{rem3.5}. Then
\begin{align*}
& d\left(  (\mathcal I_{  \vec {\mathfrak q},  \{\vec {\mathfrak q}\}'} g)(w,x) \sigma^{{\psi} }_x (\mathcal I_{{\vec q},{\{\vec q}\}'} f)(w,x) \right)   \\
= &     ({}^{\psi} {\partial}_{ \vec {\mathfrak q}, \{\vec {\mathfrak q}  \}' , r } g )(w,x) \sum _{\ell,m=0}^3 e^{\eta_{\ell,m} (w_0,\dots, x_m, \dots, w_3)}  (\mathcal I_{ \vec q , \{\vec q \}'} f)(w,x) dx \\
& +  (\mathcal I_{  \vec {\mathfrak q},  \{\vec {\mathfrak q}\}' } g)(w,x) ({}^{\psi} {\partial}_{\vec q, \{\vec q \}'} f )(w,x) \sum _{\ell,m=0}^3 e^{\lambda_{\ell,m} (w_0,\dots, x_m, \dots, w_3)} dx \\
& - \sum _{ {{   \begin{array}{c}   j,k=0, j\not=k  \\  \ell, m =0, \ell\not= m \end{array} }}}^3  \psi_j \psi_k
	(\partial^{  {\mathfrak q}_k,{\mathfrak q}_k'}_{x_k} g_j ) (w_0,\dots, x_k, \dots, w_3) e^{\eta_{\ell,m} (w_0,\dots, x_m, \dots, w_3)}   (\mathcal I_{{\vec q},{\{\vec q}\}'} f)(w,x) dx\\
& - (\mathcal I_{  \vec {\mathfrak q},  \{\vec {\mathfrak q}\}' } g)(w,x) \sum _{ {{   \begin{array}{c}   j,k=0, j\not=k  \\  \ell, m =0, \ell\not= m \end{array} }}}^3  \psi_k \psi_j
	(\partial^{q_k,q_k'}_{x_k} f_j ) (w_0,\dots, x_k, \dots, w_3) e^{\lambda_{\ell,m} (w_0,\dots, x_m, \dots, w_3)} dx,
\end{align*}
\begin{align*}
& \int_{\partial \Lambda}   (\mathcal I_{  \vec {\mathfrak q},  \{\vec {\mathfrak q}\}' } g)(w,x)     \sigma^{{\psi} }_x      (\mathcal I_{{\vec q},{\{\vec q}\}'} f)(w,x) \\
= &  \ \quad  \int_{  \Lambda}     ({}^{\psi} {\partial}_{ \vec {\mathfrak q}, \{\vec {\mathfrak q}  \}' , r } g )(w,x) \sum _{\ell,m=0}^3 e^{\eta_{\ell,m} (w_0,\dots, x_m, \dots, w_3)}  (\mathcal I_{ \vec q , \{\vec q \}'} f)(w,x) dx     \\
& + \int_{  \Lambda}  (\mathcal I_{  \vec {\mathfrak q},  \{\vec {\mathfrak q}\}' } g)(w,x)    ({}^{\psi} {\partial}_{   \vec q ,  \{\vec q \}'} f )(w,x) \sum _{\ell,m=0}^3 e^{\lambda_{\ell,m} (w_0,\dots, x_m, \dots, w_3)} dx \\
& - \int_{  \Lambda}  \sum _{ {{   \begin{array}{c}   j,k=0, j\not=k  \\  \ell, m =0, \ell\not= m \end{array} }}}^3  \psi_j \psi_k
	(\partial^{  {\mathfrak q}_k,{\mathfrak q}_k'}_{x_k} g_j ) (w_0,\dots, x_k, \dots, w_3) e^{\eta_{\ell,m} (w_0,\dots, x_m, \dots, w_3)}   (\mathcal I_{{\vec q},{\{\vec q}\}'} f)(w,x) dx
\end{align*}
\begin{align*}
& - \int_{  \Lambda}  (\mathcal I_{  \vec {\mathfrak q},  \{\vec {\mathfrak q}\}' } g)(w,x) \sum _{ {{   \begin{array}{c}   j,k=0, j\not=k  \\  \ell, m =0, \ell\not= m \end{array} }}}^3  \psi_k \psi_j
	(\partial^{q_k,q_k'}_{x_k} f_j ) (w_0,\dots, x_k, \dots, w_3) e^{\lambda_{\ell,m} (w_0,\dots, x_m, \dots, w_3)} dx,
\end{align*}
and
\begin{align*}
&  \int_{\partial \Lambda } \left[   (\mathcal I_{ \vec{\mathfrak q}, \{\vec {\mathfrak q}\}'} g)(w,\tau) \sigma_{\tau}^{\psi} K_{\psi}(\tau-x)  +  K_{\psi}(\tau-x)\sigma_{\tau}^{\psi} (\mathcal I_{{\vec q},{\{\vec q}\}'} f)(w,\tau) \right]
\nonumber \\
&  - \int_{\Lambda}  ({}^{\psi} {\partial}_{ \vec{\mathfrak  q}, \{\vec {\mathfrak  q} \}', r } g )(w,y) [\sum _{\ell,m=0}^3 e^{\eta_{\ell,m} (w_0,\dots, y_m, \dots, w_3)}] K_{\psi} (y-x)  dy \nonumber \\
&  - \int_{\Lambda}  K_{\psi} (y-x) [\sum _{\ell,m=0}^3 e^{\lambda_{\ell,m} (w_0,\dots, y_m, \dots, w_3)}]
  ({}^{\psi} {\partial}_{{\vec q},{\{\vec q}\}'} f )(w,y) dy \nonumber \\
& + \int_{\Lambda} [\sum _{ {{   \begin{array}{c}   j,k=0, j\not=k  \\  \ell, m =0, \ell\not= m \end{array} }}}^3  \psi_j \psi_k  (\partial^{{\mathfrak q}_k,{\mathfrak q}_k'}_{y_k} g_j )	(w_0,\dots, y_k, \dots, w_3)e^{\eta_{\ell,m} (w_0,\dots, y_m, \dots, w_3)}]
 K_{\psi} (y-x)  dy \nonumber \\
& + \int_{\Lambda}  K_{\psi} (y-x) [\sum _{ {{   \begin{array}{c}   j,k=0, j\not=k  \\  \ell, m =0, \ell\not= m \end{array} }}}^3  \psi_k \psi_j  (\partial^{q_k,q_k'}_{y_k} f_j )	(w_0,\dots, y_k, \dots, w_3)e^{\lambda_{\ell,m} (w_0,\dots, y_m, \dots, w_3)}] dy \nonumber \\
=  &  \left\{ \begin{array}{ll} \displaystyle  \sum _{j,k=0}^3\psi_j  (   g_j  e^{\eta_{j,k}} +  f_j  e^{\lambda_{j,k} } ) (w_0,\dots, x_k, \dots, w_3), &  x\in \Lambda,  \\ 0 , &  x\in \Omega\setminus\overline{\Lambda}.
\end{array} \right.
\end{align*}
\end{proposition}
\begin{proof}
The first two formulas are obtained using \eqref{funct} in the identities \eqref{Stokes} and combining Proposition \ref{pro3} with Remark \ref{rem3.5}. These computations, similar to that in the proof of Proposition \ref{prop}, is omitted. On the other hand, from \eqref{BPForm} we have
\begin{align*}
& \int_{\partial \Omega }    K_{\psi}(\tau-x)\sigma_{\tau}^{\psi} (\mathcal I_{{\vec q},{\{\vec q}\}'} f)(w,\tau)
 - \int_{\Omega}  K_{\psi} (y-x) {}^{\psi}\mathcal D \left[(\mathcal I_{{\vec q},{\{\vec q}\}'} f)(w,y)\right]dy
\end{align*}
\begin{align*}
=  &  \left\{ \begin{array}{ll} \displaystyle  \sum _{j,k=0}^3\psi_j f_j(w_0,\dots, x_k, \dots, w_3)  e^{\lambda_{j,k} (w_0,\dots, x_k, \dots, w_3)}
  , &  x\in \Lambda,  \\ 0 , &  x\in \Omega\setminus\overline{\Lambda},		
\end{array} \right.
\end{align*}
Equivalently, using Proposition \ref{pro3} we see that
\begin{align*}
  &  \int_{\partial \Omega } K_{\psi}(\tau-x)\sigma_{\tau}^{\psi} (\mathcal I_{{\vec q},{\{\vec q}\}'} f)(w,\tau)\\
&  - \int_{\Omega}  K_{\psi} (y-x) ({}^{\psi} {\partial}_{{\vec q},{\{\vec q}\}'} f )(w,y)   [ \sum _{\ell,m=0}^3     e^{\lambda_{\ell,m} (w_0,\dots, y_m, \dots, w_3)}] dy   \nonumber \\
  & + \int_{\Omega}  K_{\psi} (y-x)[
  \sum _{ {{   \begin{array}{c}   j,k=0, j\not=k  \\  \ell, m =0, \ell\not= m \end{array} }}}^3  \psi_k \psi_j
	(\partial^{q_k,q_k'}_{y_k} f_j )(w_0,\dots, y_k, \dots, w_3)           e^{\lambda_{\ell,m} (w_0,\dots, y_m, \dots, w_3)}
  ] dy   \nonumber \\
		=  &       \left\{ \begin{array}{ll}
\displaystyle  \sum _{j,k=0}^3\psi_j f_j(w_0,\dots, x_k, \dots, w_3)  e^{\lambda_{j,k} (w_0,\dots, x_k, \dots, w_3)}
   , &  x\in \Lambda,  \\ 0 , &  x\in \Omega\setminus\overline{\Lambda} .
\end{array} \right.
\end{align*}
The proof for the identity with $g$ is analogous to the above and finally, the sum of both formula give us the last identity.
\end{proof}
Taking $x=w$ in the last result of the Proposition \ref{proposition} yields
\begin{align*}
&  \int_{\partial \Lambda } \left[   (\mathcal I_{ \vec{\mathfrak q}, \{\vec {\mathfrak q}\}'} g)(w,\tau) \sigma_{\tau}^{\psi} K_{\psi}(\tau-w) + K_{\psi}(\tau-w)\sigma_{\tau}^{\psi} (\mathcal I_{{\vec q},{\{\vec q}\}'} f)(w,\tau) \right]
\nonumber \\
&  - \int_{\Lambda}  ({}^{\psi} {\partial}_{ \vec{\mathfrak  q}, \{\vec {\mathfrak  q} \}', r } g )(w,y)      [\sum _{\ell,m=0}^3 e^{\eta_{\ell,m} (w_0,\dots, y_m, \dots, w_3)}   ]
     K_{\psi} (y-w)  dy \nonumber \\
&  - \int_{\Lambda}  K_{\psi} (y-w) [\sum _{\ell,m=0}^3 e^{\lambda_{\ell,m} (w_0,\dots, y_m, \dots, w_3)}   ]
  ({}^{\psi} {\partial}_{{\vec q},{\{\vec q}\}'} f )(w,y) dy \nonumber \\
& + \int_{\Lambda} [\sum _{ {{   \begin{array}{c}   j,k=0, j\not=k  \\  \ell, m =0, \ell\not= m \end{array} }}}^3  \psi_j \psi_k  (\partial^{{\mathfrak q}_k,{\mathfrak q}_k'}_{y_k} g_j )	(w_0,\dots, y_k, \dots, w_3)e^{\eta_{\ell,m} (w_0,\dots, y_m, \dots, w_3)}]
 K_{\psi} (y-w) dy
\end{align*}
\begin{align*}
& + \int_{\Lambda}  K_{\psi} (y-w) [\sum _{ {{   \begin{array}{c}   j,k=0, j\not=k  \\  \ell, m =0, \ell\not= m \end{array} }}}^3  \psi_k \psi_j  (\partial^{q_k,q_k'}_{y_k} f_j )	(w_0,\dots, y_k, \dots, w_3)e^{\lambda_{\ell,m} (w_0,\dots, y_m, \dots, w_3)}] dy \nonumber \\
=  &  \left\{ \begin{array}{ll} \displaystyle  \sum _{j,k=0}^3\psi_j  (   g_j(w)  e^{\eta_{j,k}(w_0,\dots, w_k, \dots, w_3)  } +  f_j(w) e^{\lambda_{j,k} (w_0,\dots, w_k, \dots, w_3)} ) , &  w\in \Lambda,  \\ 0 , &  w\in \Omega\setminus\overline{\Lambda}.
\end{array} \right.
\end{align*}
\begin{corollary} Under the hypothesis and notations of Proposition \ref{proposition}, if moreover
$$({}^{\psi} {\partial}_{\vec{\mathfrak q}, \{\vec {\mathfrak q}\}', r}g)(w,x) = ({}^{\psi} {\partial}_{\vec q, \{\vec q \}'}f)(w,x) = 0, \quad  x\in \Omega$$
we have
\begin{align*}
& d\left(  (\mathcal I_{  \vec {\mathfrak q},  \{\vec {\mathfrak q}\}' } g)(w,x) \sigma^{{\psi} }_x (\mathcal I_{{\vec q},{\{\vec q}\}'} f)(w,x) \right)   \\
= & - \sum _{ {{   \begin{array}{c}   j,k=0, j\not=k  \\  \ell, m =0, \ell\not= m \end{array} }}}^3  \psi_j \psi_k
	(\partial^{  {\mathfrak q}_k,{\mathfrak q}_k'}_{x_k} g_j ) (w_0,\dots, x_k, \dots, w_3) e^{\eta_{\ell,m} (w_0,\dots, x_m, \dots, w_3)}   (\mathcal I_{{\vec q},{\{\vec q}\}'} f)(w,x) dx\\
& - (\mathcal I_{  \vec {\mathfrak q},  \{\vec {\mathfrak q}\}' } g)(w,x) \sum _{ {{   \begin{array}{c}   j,k=0, j\not=k  \\  \ell, m =0, \ell\not= m \end{array} }}}^3  \psi_k \psi_j
	(\partial^{q_k,q_k'}_{x_k} f_j ) (w_0,\dots, x_k, \dots, w_3) e^{\lambda_{\ell,m} (w_0,\dots, x_m, \dots, w_3)} dx,
\end{align*}
\begin{align*}
& \int_{\partial \Lambda}   (\mathcal I_{  \vec {\mathfrak q},  \{\vec {\mathfrak q}\}' } g)(w,x)     \sigma^{{\psi} }_x      (\mathcal I_{{\vec q},{\{\vec q}\}'} f)(w,x) \\
= & - \int_{  \Lambda}  \sum _{ {{   \begin{array}{c}   j,k=0, j\not=k  \\  \ell, m =0, \ell\not= m \end{array} }}}^3  \psi_j \psi_k
	(\partial^{  {\mathfrak q}_k,{\mathfrak q}_k'}_{x_k} g_j ) (w_0,\dots, x_k, \dots, w_3) e^{\eta_{\ell,m} (w_0,\dots, x_m, \dots, w_3)}   (\mathcal I_{{\vec q},{\{\vec q}\}'} f)(w,x) dx
\end{align*}
\begin{align*}
& - \int_{  \Lambda}  (\mathcal I_{  \vec {\mathfrak q},  \{\vec {\mathfrak q}\}' } g)(w,x) \sum _{ {{   \begin{array}{c}   j,k=0, j\not=k  \\  \ell, m =0, \ell\not= m \end{array} }}}^3  \psi_k \psi_j
	(\partial^{q_k,q_k'}_{x_k} f_j ) (w_0,\dots, x_k, \dots, w_3) e^{\lambda_{\ell,m} (w_0,\dots, x_m, \dots, w_3)} dx,
\end{align*}
and
\begin{align*}
&  \int_{\partial \Lambda } \left[   (\mathcal I_{ \vec{\mathfrak q}, \{\vec {\mathfrak q}\}'} g)(w,\tau)      \sigma_{\tau}^{\psi}    K_{\psi}(\tau-x)  +      K_{\psi}(\tau-x)\sigma_{\tau}^{\psi} (\mathcal I_{{\vec q},{\{\vec q}\}'} f)(w,\tau)   \right]
\nonumber \\
& + \int_{\Lambda} [\sum _{ {{   \begin{array}{c}   j,k=0, j\not=k  \\  \ell, m =0, \ell\not= m \end{array} }}}^3  \psi_j \psi_k  (\partial^{{\mathfrak q}_k,{\mathfrak q}_k'}_{y_k} g_j )	(w_0,\dots, y_k, \dots, w_3)e^{\eta_{\ell,m} (w_0,\dots, y_m, \dots, w_3)}]
 K_{\psi} (y-x)  dy \nonumber \\
& + \int_{\Lambda}  K_{\psi} (y-x) [\sum _{ {{   \begin{array}{c}   j,k=0, j\not=k  \\  \ell, m =0, \ell\not= m \end{array} }}}^3  \psi_k \psi_j  (\partial^{q_k,q_k'}_{y_k} f_j )	(w_0,\dots, y_k, \dots, w_3)e^{\lambda_{\ell,m} (w_0,\dots, y_m, \dots, w_3)}] dy \nonumber \\
=  &  \left\{ \begin{array}{ll} \displaystyle  \sum _{j,k=0}^3\psi_j  (   g_j  e^{\eta_{j,k}} +  f_j  e^{\lambda_{j,k} } ) (w_0,\dots, x_k, \dots, w_3), &  x\in \Lambda,  \\ 0 , &  x\in \Omega\setminus\overline{\Lambda}.
\end{array} \right.
\end{align*}
In particular, taking $x=w$ the previous formula becomes
\begin{align*}
&  \int_{\partial \Lambda } \left[(\mathcal I_{ \vec{\mathfrak q}, \{\vec {\mathfrak q}\}'} g)(w,\tau) \sigma_{\tau}^{\psi} K_{\psi}(\tau-w)  +      K_{\psi}(\tau-w)\sigma_{\tau}^{\psi} (\mathcal I_{{\vec q},{\{\vec q}\}'} f)(w,\tau)\right]
\nonumber \\
& + \int_{\Lambda} [\sum _{ {{   \begin{array}{c}   j,k=0, j\not=k  \\  \ell, m =0, \ell\not= m \end{array} }}}^3  \psi_j \psi_k  (\partial^{{\mathfrak q}_k,{\mathfrak q}_k'}_{y_k} g_j )	(w_0,\dots, y_k, \dots, w_3)e^{\eta_{\ell,m} (w_0,\dots, y_m, \dots, w_3)}]
 K_{\psi} (y-w)  dy \nonumber \\
\end{align*}
\begin{align*}
& + \int_{\Lambda}  K_{\psi} (y-w) [\sum _{ {{   \begin{array}{c}   j,k=0, j\not=k  \\  \ell, m =0, \ell\not= m \end{array} }}}^3  \psi_k \psi_j  (\partial^{q_k,q_k'}_{y_k} f_j )	(w_0,\dots, y_k, \dots, w_3)e^{\lambda_{\ell,m} (w_0,\dots, y_m, \dots, w_3)}] dy \nonumber \\
=  &  \left\{ \begin{array}{ll} \displaystyle  \sum _{j,k=0}^3\psi_j  (   g_j (w) e^{\eta_{j,k} (w_0,\dots, w_k, \dots, w_3)  } +  f_j(w)e^{\lambda_{j,k}(w_0,\dots, w_k, \dots, w_3) } ) , &  w\in \Lambda,  \\ 0 , &  w\in \Omega\setminus\overline{\Lambda}.
\end{array} \right.
\end{align*}
\end{corollary}
\section*{Declarations}
\subsection*{Funding} This work was partially supported by Instituto Polit\'ecnico Nacional (grant numbers SIP20232103, SIP20230312) and CONACYT.
\subsection*{Competing Interests} The authors declare that they have no competing interests regarding the publication of this paper.
\subsection*{Author contributions} All authors contributed equally to the study, read and approved the final version of the submitted manuscript.
\subsection*{Availability of data and material} Not applicable
\subsection*{Code availability} Not applicable
\subsection*{ORCID}
\noindent
Jos\'e Oscar Gonz\'alez-Cervantes: https://orcid.org/0000-0003-4835-5436\\
Juan Bory-Reyes: https://orcid.org/0000-0002-7004-1794\\
Irene Sabadini: https://orcid.org/0000-0002-9930-4308

\begin{center}
Paper  submitted in  Analysis and Mathematical Physics
\end{center}

\end{document}